\documentclass[final,leqno,onefignum,onetabnum]{siamltex1213}

\usepackage{amsmath,amssymb,exscale} 
\usepackage{amsfonts} 
\usepackage{verbatim}

\usepackage{url,doi}

\renewcommand{\ldots}{\ensuremath{\dotsc}}

\newcommand{\bgx}{{\mathcal B}}

\title{
Convergence theory for preconditioned \\eigenvalue solvers in a nutshell
\thanks{
Communicated by Nicholas Higham.
A preliminary version is posted at \url{http://arxiv.org}.
}
}
\author{
Merico E. Argentati\thanks{
Department of Mathematical and Statistical Sciences,
University Colorado Denver, P.O. Box 173364, Campus Box 170,
Denver, CO 80217-3364.
(\email{merico.argentati at ucdenver.edu}).
}
\and\
Andrew V. Knyazev\thanks{
Mitsubishi Electric Research Laboratories,
201 Broadway, Cambridge, MA 02139-1955.
(\email{andrew.knyazev at merl.com}).
}
\and\
Klaus Neymeyr\thanks{Universit\"at Rostock, Institut f\"ur Mathematik,
Ulmenstra{\ss}e 69, 18055 Rostock, Germany.
(\email{klaus.neymeyr at mathematik.uni-rostock.de and ming.zhou at mathematik.uni-rostock.de}). 
}
\and\     \\
Evgueni E. Ovtchinnikov\thanks{Numerical Analysis Group, Building R18,
STFC Rutherford Appleton Laboratory, Harwell Oxford,
Didcot, Oxfordshire, OX11 0QX, UK.
(\email{evgueni.ovtchinnikov at stfc.ac.uk}). 
}
\and\
Ming Zhou${}^\S$
}
\begin{document}
\maketitle

\begin{abstract}
Preconditioned iterative methods for numerical solution of 
large matrix eigenvalue problems
are increasingly gaining importance in various application areas, 
ranging from material sciences to data mining. 
Some of them,
e.g., those using multilevel preconditioning 
for elliptic differential operators or graph Laplacian eigenvalue problems,
exhibit almost optimal complexity in practice, 
i.e., their computational costs to calculate a fixed number of eigenvalues and eigenvectors 
grow linearly with the matrix problem size.
Theoretical justification of their optimality requires convergence rate bounds 
that do not deteriorate with the increase of the 
problem size.
Such bounds were pioneered by E.~D'yakonov over three decades ago, 
but to date only a handful have been derived, 
mostly for symmetric eigenvalue problems.
Just a few of known bounds are sharp.  
One of them is proved in
[\doi{10.1016/S0024-3795(01)00461-X}] for the simplest preconditioned eigensolver 
with a fixed step size.
The original proof has been greatly simplified and shortened in [\doi{10.1137/080727567}]
by using a gradient flow integration approach.
In the present work, we give an even more succinct
proof,
using novel ideas based on Karush-Kuhn-Tucker theory and nonlinear programming.
\end{abstract}


\bigskip

\begin{center}
 \it Dedicated to the memory of Evgenii G.~D'yakonov, Moscow, Russia, 1935--2006.
\end{center}
\bigskip


\pagestyle{myheadings}
\thispagestyle{plain}
\markboth{M.E. Argentati, A.V. Knyazev, K. Neymeyr, E.E. Ovtchinnikov, M. Zhou}
{Preconditioned eigensolver convergence in a nutshell}
\section{Introduction}

Preconditioning is a technique developed originally for
the iterative solution of linear systems
that aims at the acceleration of convergence of the iterations.
In its simplest form, 
the system $Ax = b$
is multiplied by a matrix $T$
such that the spectral condition number of $TA$,
the ratio of the largest to the smallest singular value thereof,
is considerably smaller than that of $A$,
which generally leads to faster convergence.

Iterative methods for solving linear systems normally do not require
$A$ and $T$ to be explicitly formed as matrices: it is sufficient that 
matrix-vector multiplications are implemented and performed via user-defined procedures.
The same is true for iterative methods
that compute 
eigenvalues and eigenvectors
of a very large matrix,
as, e.g.,\
in \cite{Yamada06}, calculating one eigenvector of a 100-billion size 
matrix, or in~\cite{google98}.


A classical application area for preconditioned solvers 
is the discretized boundary value problems for
elliptic partial differential operators; see, e.g.,\ \cite{DYA1996}.
With multigrid preconditioning, 
preconditioned solvers 
may achieve linear complexity on problems from this area; 
see, e.g.,\ \cite{kn2003} and references there for symmetric eigenvalue problems.
D'yakonov seminal work, summarized in \cite{DYA1996}, proposes ``spectrally equivalent'' 
preconditioning for elliptic operator eigenvalue problems in order to guarantee convergence 
that does not deteriorate with the increasing dimension of the discretized problem.
Owing to this,
for large enough problems such preconditioners
outperform direct solvers, 
which factorize
the original sparse matrix $A$.
Inevitable matrix fill-ins, 
especially prominent in discretized 
differential problems in more than two spatial dimensions, 
destroy the matrix sparsity, resulting in 
computer memory overuse and non-optimal performance. 

Preconditioning has also 
long since been
a key technique in \emph{ab initio} calculations 
in material sciences; see, e.g.,\ 
\cite{abinit2008} and references therein. 
In the last decade, preconditioning for graphs is attracting growing attention 
as a tool for achieving an optimal complexity for large data mining problems, 
e.g., for graph bisection and image segmentation using graph Laplacian and 
Fiedler vectors since \cite{k2003}; for recent work see, e.g.,\ \cite{vss14}.

Preconditioned iterative methods for the original linear system $Ax - b = 0$ are in many cases
mathematically equivalent to standard iterative methods applied for the preconditioned system 
$T(Ax - b) = 0$.
For example, the classical 
Richardson iteration step 
applied to the preconditioned system
becomes
\begin{align}
  \label{e.richls}
x_{n+1}=x_n-\tau_n T(Ax_n-b),
\end{align}
where $\tau_n$ is a suitably chosen scalar.

Turning now to eigenvalue problems, let us
consider the computation of an eigenvector
of a real symmetric positive definite matrix $A$
corresponding to its smallest eigenvalue.
Borrowing an argument from \cite{oxymoron},
suppose that the targeted eigenvalue $\lambda_*$,
or a sufficiently good approximation thereof, is known. 
Then the corresponding
eigenvector can be computed
by solving a 
homogeneous linear system $(A - \lambda_*I)x = 0$, or, 
equivalently,
the system
$T(A - \lambda_*I)x = 0$,
where $I$ is an identity. 
The Richardson iteration step 
now becomes
\begin{align}
  \label{e.richhls}
  x_{n+1}=x_n-\tau_n T(A-\lambda_*I)x_n. 
\end{align}

Theoretically,
the best preconditioners for $Ax - b = 0$ and $(A-\lambda_*I)x=0$ are, correspondingly, 
$T=A^{-1}$ and  $T=(A-\lambda_*I)^{\dagger}$, where $\dagger$ denotes a pseudo-inverse, 
making both Richardson iteration schemes, \eqref{e.richls} and \eqref{e.richhls}, 
converge in a single step with $\tau_n=1$. 
Under the standard assumption $T\approx A^{-1}$, both in \eqref{e.richls} and \eqref{e.richhls},
convergence theory is straightforward, e.g.,\ in terms of the spectral radius $\rho(I-TA)$ of $I-TA$.
Sharp explicit convergence bounds, not relying on generic constants, 
can be derived in the form of inequalities 
that allow one to determine whether the convergence deteriorates 
with the increasing problem size
by analyzing every term in the bound. 
 
For some classes of eigenvalue problems, the efficiency of choosing $T\approx A^{-1}$ 
has been demonstrated,
both numerically and theoretically, in \cite{TEV2000,KNN2003}. 
This choice allows the easy adaptation of 
a vast variety of preconditioners already developed for linear systems
to the eigensolvers.

In practice,
the theoretical value $\lambda_*$ in the Richardson iteration above
has to be replaced
with its approximation. 
A standard choice for $\lambda_*$ is a Rayleigh quotient function 
$\lambda(x)=x^T Ax/x^T x$, leading to
\begin{align}
  \label{e.rich}
  x_{n+1}=x_n-\tau_n T(A-\lambda(x_n)I)x_n. 
\end{align}

It is well known that the  Rayleigh quotient $\lambda(x_n)$ gives a high quality (quadratic) approximation 
of the eigenvalue $\lambda_*$,
if the sequence $x_n$ converges to the corresponding eigenvector. 
Thus, asymptotically as $\lambda(x_n){\,\to\,}\lambda_*,$ where $n{\,\to\,}\infty$, 
methods \eqref{e.richhls} and \eqref{e.rich} are equivalent, 
and so may be their asymptotic convergence rate bounds.  
However, asymptotic convergence rate bounds naturally contain generic constants, 
which are independent of $n\to\infty$, 
but may depend on the problem size.

Due to the changing value $\lambda(x_n)$, 
a non-asymptotic theoretical convergence analysis is much more difficult,
compared to the case for linear systems, even for the simplest methods, 
such as the Richardson iteration  \eqref{e.rich}.
D'yakonov pioneering work from the eighties, summarized in \cite[Chapter 9]{DYA1996}, 
includes the first non-asymptotic convergence bounds for preconditioned eigensolvers 
proving their linear convergence with a rate, which can be bounded above independently of the problem size.

Just a few of the known bounds are sharp. 
One of them is proved for the simplest preconditioned eigensolver with a fixed step size \eqref{e.rich} 
in a series of papers by Neymeyr over a decade ago; see \cite{KNN2003} and references therein.
The original proof has been greatly simplified and shortened in \cite{kn09}
by using a gradient flow integration approach. 

In this paper
we present a new self-contained proof of 
a sharp convergence rate bound from \cite{KNN2003}
for the preconditioned eigensolver \eqref{e.rich}, 
Theorem \ref{t.1}.
Following the geometrical approach of
\cite{KNN2003,kn09},
we reformulate the problem of finding the convergence bound for \eqref{e.rich}
as a constrained optimization problem for the Rayleigh quotient.
The main novelty of the proof is that here we use inequality constraints, which
brings to the scene the Karush-Kuhn-Tucker (KKT) theory; 
see, e.g.\ \cite{fletcher,NOW2006}. 
KKT conditions allow us to reduce 
our convergence analysis to the simplest scenario in two dimensions,
which is the key step in the proof.
We have also found several simplifications in the two dimensional convergence analysis,
compared to that of  \cite{KNN2003,kn09}.
We believe that 
the new proof 
will greatly enhance the 
understanding of the convergence behavior of increasingly 
popular preconditioned eigensolvers,
whose application area is 
quickly expanding: 
see, e.g.,\ \cite{ksu13,kps14,N2012,NOZ2011,sx14,svx15,vss14}.

\section{Convergence rate bound}

We consider a real generalized eigenvalue problem 
 $Ax=\lambda Bx$
with real symmetric
positive definite matrices $A$ and $B$.
The objective is to approximate iteratively the smallest eigenvalue $\lambda_1$
by minimizing the Rayleigh quotient
$\lambda(x)=x^T Ax/x^T Bx$.
A direct formulation of the convergence analysis with respect to this form of the
eigenvalue problem has some disadvantages. Instead, the inverted form $Bx=\mu Ax$
with $\mu=1/\lambda$ results in more compact representation of the problem
and the proof (many inverses like $A^{-1}$ and $1/\lambda$ can be avoided),
cf. \cite{KNN2003,kn09}.
For this inverted form the objective is to approximate the largest eigenvalue $\mu_1$
of $Bx=\mu Ax$ by maximizing the Rayleigh quotient $\mu(x)=x^T Bx/x^T Ax$.
We denote the eigenvalues by $\mu_1>\ldots>\mu_m>0$, which can have arbitrary multiplicity.
Corresponding eigenspaces are denoted by $\mathcal{V}_1,\ldots,\mathcal{V}_m$.

The increase of $\mu(x)$ can be achieved by 
correcting the current iterate $x$ along the preconditioned
gradient of the Rayleigh quotient, i.e.
\begin{equation}
  \label{e.preeig}
  x'=x+\frac1{\mu(x)} T(Bx-\mu(x) Ax);
\end{equation}
see \cite{KNN2003,OVT2006,kn09} and references therein.
If $B=I$, then $\mu(x)=1/\lambda(x)$ and method \eqref{e.preeig}
turns into \eqref{e.rich} with $\tau_n=1$, discussed in the Introduction. 

In all our prior work on preconditioned eigensolvers for symmetric eigenvalue problems, 
including \cite{KNN2003,kn09},
we have always assumed that the preconditioner $T$ is a symmetric and
positive definite matrix, typically satisfying conditions 
  \begin{equation}
    \label{e.basss}
(1-\gamma)z^T T^{-1}z\leq z^T Az\leq (1+\gamma) z^T T^{-1}z,
                                   \,\forall z,
\text{ for a given }\gamma\in[0,1),
  \end{equation}
or equivalent, up to the scaling of $T$. 
Recently, the authors of \cite{bdk13} have noticed and demonstrated 
that $T$ does not have to be symmetric positive definite, and
a less restrictive assumption
   \begin{equation}
    \label{e.bass}
s_{\max}\left(I-A^{1/2}TA^{1/2}\right)\leq\gamma<1,
  \end{equation}
can be used instead,
where $s_{\max}$ denotes the matrix largest singular value, 
and $A^{1/2}$ is the symmetric positive definite square root of $A$.
It is verified in \cite{bdk13}
that \eqref{e.basss} and \eqref{e.bass} are equivalent if $T$ is symmetric and
positive definite. 

In what follows,
we give a complete and concise proof 
of the following convergence rate bound, first proved in \cite{KNN2003,kn09},
\begin{theorem}
  \label{t.1}
If $\mu_{i+1}<\mu(x)<\mu_i$
and $T$ satisfies \eqref{e.bass}, 
then for $x^\prime$ given by \eqref{e.preeig} it holds
that 
either $\mu(x')\ge\mu_i$ or
  \begin{equation}
    \label{e.muest}
    0<\frac{\mu_i-\mu(x')}{\mu(x')-\mu_{i+1}}
\leq\sigma^2
    \frac{\mu_i-\mu(x)}{\mu(x)-\mu_{i+1}},\quad
\sigma=\gamma +(1-\gamma)\frac{\mu_{i+1}}{\mu_i}.
  \end{equation}
\end{theorem}

The first step, Lemma \ref{l.1}, of the proof of Theorem \ref{t.1} is the same as that in \cite{KNN2003,kn09},
where we characterize a set of possible next step iterates $x^\prime$ in \eqref{e.preeig}
varying the preconditioner $T$ constrained by assumption \eqref{e.bass}, aiming at 
eliminating the preconditioner $T$ from consideration. 
The only difference is that in \cite{KNN2003,kn09} we start with changing an original coordinate basis to an
$A$-orthogonal basis, which transforms $A$ into the identity $I$, resulting in a one-line proof of Lemma \ref{l.1}.
 Here, we choose to present a detailed proof of Lemma \ref{l.1}, for a general $A$, demonstrating that 
the transformation of $A$ into the identity $I$, made after Lemma \ref{l.1}, is well justified.

\begin{lemma}
  \label{l.1}
Let us denote $\kappa=\mu(x)$ and 
\[x'_A=A^{1/2}x',\, x_A=A^{1/2}x,\, B_A=A^{-1/2}BA^{-1/2},\, T_A=A^{1/2}TA^{1/2},\, r_A=B_A x_A-\kappa x_A,\]
 and define a closed ball
\[\bgx_A=\{y:\; \left(B_A x_A-y\right)^T \left(B_A x_A-y\right) \leq \gamma^2 \left(r_A\right)^T r_A\}\]  centered at $B_Ax_A$. 
Let $T$ satisfy \eqref{e.bass}, 
then for $x^\prime$ given by \eqref{e.preeig} it holds that
\[
\kappa x'_A = B_A x_A - (I-T_A) (B_A x_A - \kappa x_A) \in \bgx_A.
\]
\end{lemma}
\begin{proof}
Left-multiplying \eqref{e.preeig} by $\mu(x) A^{1/2}$ gives
\[\mu(x)A^{1/2}x' = \mu(x) A^{1/2} x +A^{1/2}TA^{1/2}A^{-1/2}BA^{-1/2} A^{1/2}x
- \mu(x) A^{1/2}TA^{1/2} A^{1/2} x ,\]
or, in our new notation,
\[
\kappa x'_A = \kappa x_A +T_A B_A x_A - \kappa T_A x_A
= B_A x_A - (I-T_A) (B_A x_A - \kappa x_A),
\]
resulting $B_A x_A - \kappa x'_A =  (I-T_A) r_A$. 
Since $s_{\max}\left(I-T_A\right)\leq\gamma$ by \eqref{e.bass}, we get
$
\left(B_A x_A - \kappa x'_A\right)^T \left(B_A x_A - \kappa x'_A\right) = 
\left((I-T_A) r_A\right)^T \left((I-T_A) r_A\right)\leq 
\gamma^2  \left(r_A\right)^T r_A.
$
\end{proof}

The second step of the proof is traditional---reducing the generalized symmetric eigenvalue problem 
$Bx=\mu Ax$ to the standard eigenvalue problem for the symmetric positive definite matrix  $B_A=A^{-1/2}BA^{-1/2}$
by making the change of variables as hinted by Lemma \ref{l.1}. 
We use the standard inner product in  $\cdot_A$ variables, i.e.
\[(y_A,z_A)=y_A^Tz_A=\left(A^{1/2}y\right)^T\left(A^{1/2}z\right)=y^TAz,\]
and the corresponding vector norm  $\|y_A\|=(y_A,y_A)^{1/2}$, so, e.g.,\
\[
\kappa=\mu(x)=\frac{x^TBx}{x^TAx}=\frac{x_A^T B_A x_A}{x_A^T x_A}= 
\frac{\left(x_A, B_A x_A\right)}{\left(x_A, x_A\right)}=\mu_A(x_A).
\]
We later use $B_A$ and $B_A^{-1}$ -based scalar products and norms defined as follows, e.g.,\ 
\[
(y_A,z_A)_{B_A}=y_A^T B_A z_A =\left(A^{1/2}y\right)^TA^{-1/2}BA^{-1/2}\left(A^{1/2}z\right)=y^T Bz.
\]
For brevity we drop the subscript $\cdot_A$ in the rest of the paper.
In the following $T$ refers to $T_A=A^{1/2}TA^{1/2}$, $B$ refers to $B_A=A^{-1/2}BA^{-1/2}$, 
$x$ refers to $x_A=A^{1/2}x$, and so on, cf. Lemma \ref{l.1}.
Furthermore, $\mu(x)={(x,Bx)}/{(x,x)}$, and method \eqref{e.preeig} is
$\mu(x)x' = Bx-(I-T)(Bx-\mu(x)x)$. The new form of condition \eqref{e.bass}
is $\|I-T\|\leq\gamma$. This means that $A^{1/2}TA^{1/2}$ approximates the identity matrix
with respect to the notation used in Lemma \ref{l.1}.
The closed ball has the form $\bgx=\{y:\; \|Bx-y\|\leq \gamma{\|r\|}\}$ with the radius
$\gamma{\|r\|}$ centered at $Bx$.
Since $\mu(x)x'\in\bgx$ and $\mu(x')=\mu\big(\mu(x)x'\big)$,
we can estimate $\mu(x')$ by using a minimizer of $\mu(\cdot)$ in $\bgx$
(i.e. by considering the worst case).
We observe that, effectively, we set $A=I$ without loss of generality. 

\section{The special case with $\gamma{\,=\,}0\,$: the power method}\label{s.cc}
\label{s3}
The main idea of the geometrical approach of \cite{KNN2003,kn09},
which we also employ in this paper,
is that the convergence rate of iterations \eqref{e.preeig}
is slowest, in terms of the  Rayleigh quotient, if 
$x$ is a linear combination of
two eigenvectors,
which makes the further convergence analysis trivial.
A new proof of this fact actually
occupies a major part of our paper.
In order to illustrate how such a dramatic reduction in 
dimension becomes possible, in this section we 
apply our technique to
a simplified case $T=I$ corresponding to $\gamma=0$.
It is not difficult to see that under this assumption
\eqref{e.preeig} turns into one iteration $\mu(x)x'=Bx$ of the power method,
and bound \eqref{e.muest} holds with  $\gamma=0$ and thus $\sigma={\mu_{i+1}}/{\mu_i}$.
Let us make a historic note that exactly this result has apparently first appeared in
\cite{KNY1986,KNY1987}.
  
The left-hand side of bound \eqref{e.muest} is monotone in $\mu\left(x'\right)=\mu(Bx)$.
One way
to find out at which $x$
the behavior of \eqref{e.preeig} is the worst
is to minimize
$f(x)=\mu(Bx)$ 
for all $x$ that satisfy $\mu(x)=\kappa$
for some fixed $\kappa \in (\mu_{i+1},\mu_i)$.
Slightly abusing the notation in the proof, 
we keep denoting by $x$ both the initial approximation  in \eqref{e.preeig} 
and the vector in the minimization problem. 

We notice that $\mu(x)=\kappa$
is equivalent to
$h(x)=\kappa(x,x)-(x,Bx)=0$.
Therefore,
at a stationary point we have, using Lagrangian multipliers, that  
\begin{equation}\label{classic_lagrangian}
\nabla f(x)+a \nabla h(x)=0,
\end{equation}
where $a$ is some constant.
This yields
\[
\frac{2 B ( B - \mu(B x) I) B x}{\|Bx\|^2}+ 2a(\kappa x-Bx)=0,
\]
which can be rewritten as
\begin{equation}
\label{cubic.eq.pm}
B^3 x - \mu(B x) B^2 x - c B x + c\kappa x = 0,
\end{equation}
%
where $c=\|Bx\|^2a$.
Since 
\[
\mu(Bx) = \frac{\big(Bx,B(Bx)\big)}{\big(Bx,Bx\big)}
\]
implies $(B^3 x - \mu(B x) B^2 x, x) = 0$, we obtain
\[\begin{split}
& c\|B x - \kappa x\|^2 =
(B^3 x - \mu(B x) B^2 x, B x - \kappa x) 
\\
=\ & (B^3 x - \mu(B x) B^2 x, B x - \mu(B x)x) =
\|B^2 x - \mu(B x) B x\|^2,
\end{split}\]
which shows that $c > 0$.
Thus, equation \eqref{cubic.eq.pm}
can be viewed as a polynomial equation $p_3(B)x = 0$, where $p_3(t)$ is a third degree polynomial 
with positive first and last coefficients, specifically $1$ and $c \kappa$, correspondingly.

Inserting $x=\sum_{i=1}^m v_i$, where
$v_i$ are the projections
of $x$ onto the eigenspaces $\mathcal{V}_i$, leads to
$\sum_{i=1}^m p_3(\mu_i)v_i = 0$. 
Since the eigenspaces are orthogonal to each other,
the products $p_3(\mu_i)v_i$ must be zero for each $i$.
Owing to the positiveness of the first and last coefficients,
the polynomial $p_3$ must have a non-positive root, 
and thus at most two positive roots,
i.e. $p_3(\mu_i)$ can be zero for some two indexes $i=k$  and $i=l$ at most, allowing 
the only possibly nonzero $v_k$ and $v_l$ from all projections $v_i$.
We conclude that $x$ is a linear combination
of at most two normalized eigenvectors $x_k$ and $x_l$, 
corresponding to distinct eigenvalues $\mu_k$ and $\mu_l$ of the matrix $B$.

We assume without loss of generality 
that $x=x_k+\alpha x_l$, then 
\[
\alpha^2=\frac{\mu_k-\mu(x)}{\mu(x)-\mu_l}=\tan^2\angle\left(x,x_k\right).\]
Similarly, since $Bx=\mu_k x_k + \alpha \mu_l x_l$, we obtain
\begin{equation}
\label{2D.eq.pm}
\tan^2\angle\left(Bx,x_k\right)=\frac{\mu_k-\mu(Bx)}{\mu(Bx)-\mu_l}=\frac{\mu_l^2}{\mu_k^2}\alpha^2
=\sigma^2 \frac{\mu_k - \mu(x)}{\mu(x) - \mu_l},
\quad\mbox{with}\quad
\sigma = \frac{\mu_l}{\mu_k}.
\end{equation}
Let $\mu_k > \mu_l$, then 
$\kappa\in(\mu_l,\mu_k)$ implies $\mu_l\leq\mu_{i+1}<\mu(x)=\kappa<\mu_i\leq\mu_k$.
By using monotonicity of the ratio of the quotients in $\mu_k$ and $\mu_l$
and the fact that the vector $x$ here corresponds to the worst-case scenario, i.e. 
minimizing $\mu\left(x'\right)=\mu(Bx)$ over all $x$ with the fixed value $\mu(x)=\kappa$, 
we obtain \eqref{e.muest} with $\gamma=0$. Since \eqref{2D.eq.pm} is an equality, we also prove that the upper 
bound in \eqref{e.muest} with $\gamma=0$ is sharp, turning into an equality if 
the initial approximation in \eqref{e.preeig} satisfies $x\in{\rm span}\left\{x_i,\,x_{i+1}\right\}$.

In the next section, we apply the described dimensionality reduction technique
to the general case.
We formulate the conditions that ``the worst case'' $x$ must satisfy,
which yield the generalization of equation \eqref{classic_lagrangian},
and rewrite this equation as a cubic equation $p_3(B)x = 0$.
We show that the first and last coefficients of this equation are positive,
which, as we have just seen, implies that $x$ is a linear combination of
two eigenvectors.
A simple two-dimensional analysis completes the proof of Theorem~\ref{t.1}.

\section{The general case with $\gamma\in[0,1)\,$: the preconditioned eigensolver \eqref{e.preeig}}
\label{s.geo} \label{s4}
Next the proof of Theorem \ref{t.1} is given:
Let us denote $r=Bx-\mu(x)x$ and define 
$\bgx=\{y:\; \|Bx-y\|\leq \gamma{\|r\|}\}$, a closed ball with the radius
$\gamma{\|r\|}$ centered at $Bx$.
On the one hand, it holds that $\mu(y)>\mu(x)$ for any vector $y\in\bgx$ since $x$ is not an eigenvector and
$r\neq0$. Indeed, taking into account 
$\|Bx-y\|^2<\|r\|^2=\|Bx-\mu(x)x\|^2$,
we have
\[\begin{split}\|y\|^2
&<2(x,y)_B-2\mu(x)\|x\|_B^2+\mu(x)^2\|x\|^2=2(x,y)_B-\mu(x)\|x\|_B^2\\
&=\big(\|y\|_B^2-\|y-\mu(x)x\|_B^2\big)/\mu(x)\le\|y\|_B^2/\mu(x)\\
\end{split}\]
so that $\mu(x)<\|y\|_B^2/\|y\|^2=\mu(y)$.
On the other hand, $\mu(x)x'=Bx-(I-T)r\in\bgx$,
since $\|(I-T)r\|\leq\gamma\|r\|$, see Lemma \ref{l.1}.
This proves $\mu(x')>\mu(x)>\mu_{i+1}$ and, thus, the left inequality in \eqref{e.muest}, provided that $\mu(x')<\mu_i.$

In the previous proof with $\gamma=0$, the ball $\bgx$ shrinks to a single point $Bx$ and the only choice of $y=Bx$ is possible. 
The present case  $\gamma>0$ is significantly more difficult for the worst-case scenario analysis, involving a minimization problem 
with two variables, $x$ and $y$. In our previous work, see \cite{KNN2003,kn09} and references therein, 
we first vary $y\in\bgx$ intending to minimize $\mu(y)$ for a given $x$, and then vary $x$ fixing $\mu(x)=\kappa$.
The first minimization problem defines $y$ as an implicit function of $x$, and then  
Lagrangian multipliers are used, as in Section \ref{s3}, to analyze the second minimization problem, in $x$. 
It turns out that the proof is much simpler if we vary both $x$ and $y$ at the same time
and attack the required two-parameter minimization problem in $x$ and $y$ directly by using the KKT arguments as provided below.

\begin{lemma}\label{t.reduc}
For $\gamma\in[0,1)$ and a fixed value $\kappa$
that is not an eigenvalue of $B$, let a pair of vectors $\{x^*,y^*\}$
denote a solution of the following constrained minimization problem:
\[\text{minimize}\ \ \mu(y)\ \
\text{subject to} \ \ \|Bx-y\| \le \gamma \|Bx-\kappa x\|\ \ \text{and} 
                 \ \ \mu(x)=\kappa.\]
If $x^*$ is not an eigenvector of $B$,
then both $x^*$ and $y^*$ belong to a two-dimensional
invariant subspace of $B$ corresponding to two distinct eigenvalues, and
\begin{equation}\label{sineq}
\sin\angle(Bx^*,y^*)=\gamma\sin\angle(Bx^*,x^*),
\end{equation}
where $\angle(\cdot,\cdot)$ denotes an angle between two vectors
defined by $\angle(u,v):=\arccos\left(\dfrac{(u,v)}{\|u\|\|v\|}\right)$.
\end{lemma}

\begin{proof}
We consider the equivalent problem
\begin{align*}
\text{minimize}   & \ \ f(x,y)=\mu(y),\, x\neq0, \\
\text{subject to} & \ \ g(x,y)=\|Bx-y\|^2 -\gamma^2 \|Bx-\kappa x\|^2\leq 0, \text{ and} \\
                  & \ \ h(x,y)=\kappa(x,x)-(x,Bx)=0.
\end{align*}
We first notice that the assumption $x\neq0$ implies $y\neq0$ because of the first constraint and
$\gamma<1$. Thus, $\mu(y)$ is correctly defined. Next, let us temporarily
consider a stricter constraint  $\|x\|=1$, instead of $x\neq0$. Combined with the other
constraints, this results in minimization of the smooth function $f(x,y)$ on a compact set, so there
exists a solution $\{x^*,y^*\}$. Finally, let us remove the artificial constraint  $\|x\|=1$ and
notice that any nonzero multiple of $\{x^*,y^*\}$ is also a solution. Thus we can consider the
Karush-Kuhn-Tucker (KKT) conditions, e.g., \cite[Theorem 9.1.1]{fletcher}, \cite{fletcher,NOW2006},
in any neighborhood of  $\{x^*,y^*\}$, which does not include the origin.

Next we show that the gradients of $g$ and $h$ are linearly independent. 
For the gradient of $h$, it holds that $\partial h/\partial x=-2r$ with $r\neq0$,
since $x$ is not an eigenvector of $B$, and it holds that $\partial h/\partial y=0$,
since $h$ does not depend on $y$.
Assuming the linear dependence of the gradients of $g$ and $h$ implies that
$\partial g/\partial y=0$, so that $2(y-Bx)=0$ and $y=Bx$.
By using $y=Bx$, it holds that
\[
\frac{\partial g}{\partial x}=2\left(B^2x-By-\gamma^2(B-\kappa I)r\right) = -2 \gamma^2(B-\kappa I)r, 
\]
while $\partial h/\partial x=-2r$, i.e. 
(using again the assumed linear dependence)
the vector $r=Bx-\kappa x$ is an eigenvector of $B-\kappa I$, and,
hence $x$ is an eigenvector of $B$, contradicting the lemma assumption. 

Therefore, the gradients of $g$ and $h$ are linearly independent.
All functions involved in our constrained minimization are smooth.
We conclude that the stationary point $\{x^*,y^*\}$ is regular, i.e.,\ the KKT conditions are valid.
The KKT stationarity condition states that there exist constants
$a$ and $b$ such that
\[
\nabla f(x^*,y^*)+a \nabla g(x^*,y^*)+b \nabla h(x^*,y^*)=0
\]
at the critical point  $\{x^*,y^*\}$.
The independent variables  $\{x,y\}$ no longer appear, so
to simplify the notation, in the rest of the proof we
drop the superscript $^*$ and substitute  $\{x,y\}$ for  $\{x^*,y^*\}$.
We separately write the partial derivatives with respect to $x$,
\begin{equation}\label{partialx}
2a \left(B^2x-By-\gamma^2(B-\kappa I)r \right) - 2br=0,
\end{equation}
and with respect to $y$,
\begin{equation}\label{partialy}
2\frac{By-\mu(y)y}{(y,y)}+2a(y-Bx)=0.
\end{equation}
The KKT complementary slackness condition
$a g(x,y)=0$ must be satisfied, implying 
\begin{equation}\label{cslackness}
\|Bx-y\| =\gamma \|r\| \text{ if } a\neq0. 
\end{equation}

If $y$ is an eigenvector then $By-\mu(y)y=0$ in condition \eqref{partialy}, leading to $y=Bx$, i.e. vector $x$ is also an eigenvector of $B$, thus we are done.  Now we consider a nontrivial case, where neither $x$ nor $y$ is an eigenvector.
Condition \eqref{partialy} then implies $a\neq0$, so identity \eqref{cslackness} holds unconditionally, 
condition \eqref{partialx} turns into
\begin{equation}\label{residual_formula}
B(Bx-\gamma^2r-y)=c r \qquad \text{ with } c=\frac{b}{a} - \gamma^2\kappa,
\end{equation} 
and  taking the inner product of \eqref{partialy} with $y$ gives
\begin{equation}\label{Bxmyy}
 (Bx-y,y)=0.
\end{equation} 

Taking the inner products of both sides of \eqref{residual_formula}
with $B^{-1}r$ results in
\begin{align*}
c\|r\|_{B^{-1}}^2 =
 (B x - y, r)-\gamma^2\|r\|^2 =
 (B x - y, r)-\|B x - y\|^2 =
 -\kappa(B x - y, x).
\end{align*}
Therein we use \eqref{cslackness} and \eqref{Bxmyy}. 

Denoting $d=a\|y\|^2-\mu(y)$,
we rewrite \eqref{partialy} as
\begin{equation}\label{bxy2}
By-\mu(y)Bx=d(Bx-y).
\end{equation}
Taking the inner products of both sides of \eqref{bxy2} with $y-\mu(y)x$ 
yields
\begin{eqnarray*}
0\leq \|y - \mu(y) x\|_B^2 = \
d (B x - y, y - \mu(y) x) =
-d \mu(y) (B x - y, x),
\end{eqnarray*}
where the orthogonality $(Bx-y,y)=0$ 
has been used again.
Therefore, we obtain 
$cd\|r\|_{B^{-1}}^2=-d\kappa(Bx-y,x)\geq0$,
which implies $cd\geq0$.

Substituting $r=Bx-\kappa x$ and multiplying through by $B$
in \eqref{residual_formula} results in
\[
(1-\gamma^2) B^3 x+(\kappa\gamma^2-c)B^2 x-B^2 y+c\kappa Bx=0.
\]
Multiplying through by $d+\mu(y)$ and substituting $(d+\mu(y))Bx=(B+ d I)y$,
which follows from \eqref{bxy2}, we obtain
$p_3(B)y=(c_3B^3+c_2B^2+c_1B+c_0)y=0$, where $p_3(\cdot)$ is a third degree polynomial
with $c_3=1-\gamma^2>0$ and $c_0=cd\kappa\geq0$,
which cannot have more than two positive roots.
Thus, $y$ is a linear combination
of two normalized eigenvectors $x_k$ and $x_l$ 
corresponding to two distinct eigenvalues
(cf. Section~3),
i.e.
$y\in \mathcal Z:=\mathrm{span}\{x_k,x_l\}$. 
Since
$d+\mu(y)=a\|y\|^2\neq0$, by \eqref{bxy2} so is $x$.

Furthermore, the orthogonality $(Bx-y,y)=0$ from \eqref{Bxmyy}
shows that
\[\cos^2\angle(Bx,y)=\frac{(Bx,y)^2}{\|Bx\|^2\|y\|^2}
=\frac{(y,y)^2}{\|Bx\|^2\|y\|^2}=\frac{\|y\|^2}{\|Bx\|^2},\]
and $\sin^2\angle(Bx,y)=1-\cos^2\angle(Bx,y)
=(\|Bx\|^2-\|y\|^2)/\|Bx\|^2=\|Bx-y\|^2/\|Bx\|^2$.
This leads to $\sin\angle(Bx,y)=\|Bx-y\|/\|Bx\|$,
since the angles between vectors have the range $[0,\pi]$
(due to $\arccos$).
Similarly, $(Bx-\kappa x,x)=0$ together with $\kappa>0$ implies
$\sin\angle(Bx,x)=\sin\angle(Bx,\kappa x)=\|Bx-\kappa x\|/\|Bx\|$.
Then we have $\sin\angle(Bx,y)=\gamma\sin\angle(Bx,x)$
by using \eqref{cslackness}.
\end{proof}

We now derive bound \eqref{e.muest} in a two-dimensional $B$-invariant subspace.
\begin{lemma}\label{t.2D}
Let $x^*$ and $y^*$ belong to a two-dimensional
invariant subspace of $B$ corresponding to the eigenvalues
$\mu_k>\mu_l$ and satisfy \eqref{sineq}, where $x^*$ is not an eigenvector.
It holds that
\begin{align}
   \label{e.2Dest}
   \frac{\mu_k-\mu(y^*)}{\mu(y^*)-\mu_l}\;
   \frac{\mu(x^*)-\mu_l}{\mu_k-\mu(x^*)}
   \le\left(\gamma+(1-\gamma)\frac{\mu_l}{\mu_k}\right)^2.
\end{align}
\end{lemma}
\begin{proof}
In this proof we drop the superscript $^*$ upon
$x$ and $y$.
The vectors $x$, $y$ and $Bx$ can be represented by the coefficient vectors
\[u:=c_1(1,\alpha)^T,\quad v:=c_2(1,\beta)^T \text{ \;and \;} w:=c_1(\mu_k,\alpha\mu_l)^T\]
with respect to an orthonormal basis $\{x_k,x_l\}$,
where $x_k$ and $x_l$ are eigenvectors associated with $\mu_k$ and $\mu_l$.
Evidently, it holds that $(Bx,y)=(w,v)$, $\|Bx\|=\|w\|$, $\|y\|=\|v\|$
by using the orthonormal basis. Therefore, $\angle(Bx,y)=\angle(w,v)$,
and similarly $\angle(Bx,x)=\angle(w,u)$.
This allows us to rewrite \eqref{sineq}
in the form $\,\sin\angle(w,v)=\gamma\sin\angle(w,u)$. 
Using the geometric property of the cross products
\[\tilde{v}:=\begin{bmatrix}w\\0\end{bmatrix}\times\begin{bmatrix}v\\0\end{bmatrix}
\quad\mbox{and}\quad
\tilde{u}:=\begin{bmatrix}w\\0\end{bmatrix}\times\begin{bmatrix}u\\0\end{bmatrix},\]
we have
\[\frac{\|\tilde{v}\|}{\|w\|\,\|v\|}=\gamma\,\frac{\|\tilde{u}\|}{\|w\|\,\|u\|},\]
which yields
\[\gamma^2=\frac{\|\tilde{v}\|^2\|u\|^2}{\|\tilde{u}\|^2\|v\|^2}
=\frac{(\beta\mu_k-\alpha\mu_l)^2(1+\alpha^2)}{(\alpha\mu_k-\alpha\mu_l)^2(1+\beta^2)}.\]
Further, we use the equalities
\[\alpha^2=\tan^2\angle(x,x_k)=\frac{\mu_k-\mu(x)}{\mu(x)-\mu_l},\quad
\beta^2=\tan^2\angle(y,x_k)=\frac{\mu_k-\mu(y)}{\mu(y)-\mu_l},
\]
which can be derived in a similar way to Section \ref{s.cc}.
Then $\dfrac{1+\alpha^2}{1+\beta^2}=\dfrac{\mu(y)-\mu_l}{\mu(x)-\mu_l} \ge 1$,
so that
\[\gamma^2\ge\frac{(\beta\mu_k-\alpha\mu_l)^2}{(\alpha\mu_k-\alpha\mu_l)^2},
\quad\mbox{and}\quad\left|\frac{\beta}{\alpha}-\frac{\mu_l}{\mu_k}\right|
\le\gamma\left|1-\frac{\mu_l}{\mu_k}\right|.\]
Since $0<\mu_l/\mu_k<1$, we have
\[\left|\frac{\beta}{\alpha}\right|
\le\left|\frac{\beta}{\alpha}-\frac{\mu_l}{\mu_k}\right|+\left|\frac{\mu_l}{\mu_k}\right|
\le\gamma\left(1-\frac{\mu_l}{\mu_k}\right)+\frac{\mu_l}{\mu_k}
=\gamma+(1-\gamma)\frac{\mu_l}{\mu_k},\]
which proves \eqref{e.2Dest} by using
$\dfrac{\mu_k-\mu(y)}{\mu(y)-\mu_l}\;\dfrac{\mu(x)-\mu_l}{\mu_k-\mu(x)}
=\left|\dfrac{\beta}{\alpha}\right|^2$.
\end{proof}

The proof of Theorem \ref{t.1} is completed by deriving 
convergence bound \eqref{e.muest} from its two-dimensional version.
We restate the assumption $\mu_{i+1}<\mu(x)<\mu(x')<\mu_i$.
According to Lemma \ref{t.reduc}, there exists a minimizer $y$ with $\mu(y)\le\mu(x')$, which
satisfies \eqref{e.2Dest} with $\mu_l<\mu(x)<\mu(y)<\mu_k$, and the
interval $(\mu_{i+1},\mu_i)$ is a subset of $(\mu_l,\mu_k)$. 
Using monotonicity arguments,  \eqref{e.2Dest}, 
and the same arguments as Section \ref{s.cc}, we obtain 
\[\begin{split}
\frac{\mu_i-\mu(y)}{\mu(y)-\mu_{i+1}}\;
&  \frac{\mu(x)-\mu_{i+1}}{\mu_i-\mu(x)}
=\frac{\mu_i-\mu(y)}{\mu_i-\mu(x)}\;
   \frac{\mu(x)-\mu_{i+1}}{\mu(y)-\mu_{i+1}}
\le\frac{\mu_k-\mu(y)}{\mu_k-\mu(x)}\;
   \frac{\mu(x)-\mu_l}{\mu(y)-\mu_l}\\[1ex]
&=\frac{\mu_k-\mu(y)}{\mu(y)-\mu_l}\;
   \frac{\mu(x)-\mu_l}{\mu_k-\mu(x)}
   \le\left(\gamma+(1-\gamma)\frac{\mu_l}{\mu_k}\right)^2
\le\left(\gamma+(1-\gamma)\frac{\mu_{i+1}}{\mu_i}\right)^2.
\end{split}\]
This proves \eqref{e.muest} since
$(\mu_i-\mu(x'))/(\mu(x')-\mu_{i+1})\le(\mu_i-\mu(y))/(\mu(y)-\mu_{i+1})$.

\section*{Conclusions}
\label{secC}
We have presented a succinct
proof of the standard sharp convergence rate bound
for the simplest fixed step size preconditioned eigensolver. 
The key argument of the new proof is 
the characterization of the case of poorest convergence
as a constrained optimization problem for the Rayleigh quotient.
Employing the
Karush-Kuhn-Tucker conditions
and some elementary matrix algebra, we
dramatically simplify
the convergence analysis 
by reducing it to
a 
subspace spanned by two eigenvectors.
We expect the 
analytical framework 
developed
in this paper to be a valuable tool 
in the convergence analysis of
a variety of preconditioned eigensolvers;
see, e.g., the analysis of the preconditioned steepest descent method
in \cite{N2012}.

\section*{Appendix}
\textbf{I}. An alternative estimate for the left-hand side of \eqref{e.2Dest},
which is sharp with respect to all variables, can be derived as follows:
With $\delta:=\beta/\alpha$ and $\varepsilon:=\mu_l/\mu_k$,
a new representation of $\gamma^2$ is given by
\[\gamma^2
=\frac{(\delta-\varepsilon)^2(1+\alpha^2)}{(1-\varepsilon)^2(1+\alpha^2\delta^2)}.\]
This results in a quadratic equation for $\delta$ with the roots
\[\delta_{\pm}=\frac{\varepsilon(1+\alpha^2)\pm\gamma(1-\varepsilon)
\sqrt{(1+\alpha^2)(1+\alpha^2\varepsilon^2)-\alpha^2\gamma^2(1-\varepsilon)^2}}
{(1+\alpha^2)-\alpha^2\gamma^2(1-\varepsilon)^2}.\]
Since $\delta^2=\dfrac{\beta^2}{\alpha^2}=\dfrac{\mu_k-\mu(y)}{\mu(y)-\mu_l}\;
   \dfrac{\mu(x)-\mu_l}{\mu_k-\mu(x)}$,
a strictly sharp bound for the estimate in \eqref{e.2Dest}
is given by $\max\{\delta_{+}^2,\delta_{-}^2\}=\delta_{+}^2$.
We note that in the limit case $\mu(x)\to\mu_k$ it holds that $\alpha\to0$,
and $\delta_{+}^2$ turns into $(\varepsilon+\gamma(1-\varepsilon))^2$.  This
coincides with the known bound in \eqref{e.2Dest}.\\
\textbf{II}.
The bound in \eqref{e.2Dest} contains a convex combination of $1$ and $\mu_l/\mu_k$.
Interestingly, this bound can also be derived by using a convex function as follows:
Without loss of generality, we assume that $x$ has a positive $x_k$ coordinate.
Then $\angle(x,x_k)$ is an acute angle. Since $B>0$, $\angle(Bx,x)$ and $\angle(Bx,x_k)$
are also acute angles. The equality \eqref{sineq}
together with $\gamma<1$ shows further $\angle(Bx,y)<\angle(Bx,x)<\pi/2$.
Since $\angle(Bx,x_k)$ and $\angle(Bx,y)$ are acute angles,
the vectors $x_k$ and $y$ are located in a half-plane
whose boundary line is orthogonal to $Bx$.
A simple case differentiation shows
that $\angle(y,x_k)$ is either equal to $|\angle(Bx,x_k)-\angle(Bx,y)|$
or equal to $\angle(Bx,x_k)+\angle(Bx,y)$.
Further, we use the equalities
\[\tan^2\angle(y,x_k)=\frac{\mu_k-\mu(y)}{\mu(y)-\mu_l},\;\;
\tan^2\angle(x,x_k)=\frac{\mu_k-\mu(x)}{\mu(x)-\mu_l},\;\;
\frac{\tan^2\angle(Bx,x_k)}{\tan^2\angle(x,x_k)}=\frac{\mu_l^2}{\mu_k^2},
\]
which can be derived in a similar way to Section \ref{s.cc}.
The last equality proves $\angle(Bx,x_k)<\angle(x,x_k)$,
since the tangent is an increasing function for acute angles, and $\mu_l<\mu_k$.
This leads to $\angle(x,x_k)=\angle(Bx,x_k)+\angle(Bx,x)$,
since $x, Bx, x_k$ are all in the same quadrant.
In summary, it holds that 
\begin{equation}\label{e.2Dest1a}
\angle(y,x_k)\le\angle(Bx,x_k)+\angle(Bx,y)<\angle(Bx,x_k)+\angle(Bx,x)
=\angle(x,x_k)<\pi/2,
\end{equation}
i.e., $\angle(y,x_k)$ is a further acute angle.
Using these acute angles, we write \eqref{e.2Dest} equivalently as
\begin{equation}\label{e.2Dest1}
\tan\angle(y,x_k)\le\gamma\tan\angle(x,x_k)
+(1-\gamma)\tan\angle(Bx,x_k).
\end{equation}
In order to prove \eqref{e.2Dest1}, we use \eqref{sineq} again,
together with
$\varphi:=\angle(Bx,x)$, $\vartheta:=\angle(Bx,x_k)$
and the first inequality in \eqref{e.2Dest1a}.
It holds that
\[\tan\angle(y,x_k)\le\tan[\vartheta+\arcsin\big(\gamma\sin(\varphi)\big)]=:f(\gamma).\]
Because of
\[
f'(\gamma)=
\frac
{\big(1+f(\gamma)^2\big)\sin(\varphi)}
{\sqrt{1-\big(\gamma\sin(\varphi)\big)^2}}\ge0
\quad\mbox{for}\quad\gamma\in[0,1],
\]
$f(\gamma)$ is a monotonically increasing function in $[0,1]$. 
The numerator of $f'(\gamma)$ is also a monotonically increasing
function and its denominator is monotonically decreasing in $\gamma\in[0,1]$. 
These two functions are positive so that
$f'(\gamma)$ is also a monotonically increasing function.
Thus $f(\gamma)$ is a convex function in $[0,1]$, and
\[
\tan\angle(y,x_k)\le f(\gamma)\le(1-\gamma)f(0)+\gamma f(1)
=(1-\gamma)\tan(\vartheta)+\gamma\tan(\vartheta+\varphi),
\]
which proves \eqref{e.2Dest1} and hence \eqref{e.2Dest}.

%

\def\refname{\centerline{\footnotesize\rm REFERENCES}}
\def\bibindent{0ex}
\makeatletter 
\renewcommand\@biblabel[1]{#1.\ } 
\makeatother

\end{document}